\documentclass{article}  
\usepackage{amsmath,amsthm,amssymb, fullpage, graphicx, url}
\usepackage{dsfont} 
\usepackage[mathlines]{lineno}
\usepackage{tikz}

\def\picAAA{
\begin{tikzpicture}
\draw[thick]
    (0,0)--(0,1)--(1,1)--(1,0)--(0,0)
    (-0.7071,0.7071)--(-0.7071,1.7071)--(0.2929,1.7071)--(0.2929,0.7071)--(-0.7071,0.7071)
    (0.7071,0.7071)--(0.7071,1.7071)--(1.7071,1.7071)--(1.7071,0.7071)--(0.7071,0.7071)
    (0,1.4142)--(0,2.4142)--(1,2.4142)--(1,1.4142)--(0,1.4142)
    (0,0)--(-0.7071,0.7071)--(0,1.4142)--(0.7071,0.7071)--(0,0)
    (1,0)--(0.2929,0.7071)--(1,1.4142)--(1.7071,0.7071)--(1,0)
    (0,1)--(-0.7071,1.7071)--(0,2.4142)--(0.7071,1.7071)--(0,1)
    (1,1)--(0.2929,1.7071)--(1,2.4142)--(1.7071,1.7071)--(1,1)
    (-1.7071,0.7071)--(-0.7071,0.7071);
\draw[thick,fill=white]
    (0,0) circle (2pt)
    (0,1) circle (2pt)
    (1,1) circle (2pt)
    (1,0) circle (2pt)
    (-0.7071,0.7071) circle (2pt)
    (-0.7071,1.7071) circle (2pt)
    ( 0.2929,1.7071) circle (2pt)
    ( 0.2929,0.7071) circle (2pt)
    (0.7071,0.7071) circle (2pt)
    (0.7071,1.7071) circle (2pt)
    (1.7071,1.7071) circle (2pt)
    (1.7071,0.7071) circle (2pt)
    (0,1.4142) circle (2pt)
    (0,2.4142) circle (2pt)
    (1,2.4142) circle (2pt)
    (1,1.4142) circle (2pt)
    (-1.7071,0.7071) circle (2pt);
\end{tikzpicture}
}

\def\picBBB{
\begin{tikzpicture}
\draw[thick]
    (90:1)--(150:1)--(210:1)--(270:1)--(330:1)--(30:1)--(90:1)
    (90:1)--(150:0.5)--(210:1)--(270:0.5)--(330:1)--(30:0.5)--(90:1)
    (150:0.5)--(150:1)
    (270:0.5)--(270:1)
    (30:0.5)--(30:1);
\draw[thick,fill=white]
    (90:1) circle (2pt)
    (150:1) circle (2pt)
    (150:0.5) circle (2pt)
    (210:1) circle (2pt)
    (270:1) circle (2pt)
    (270:0.5) circle (2pt)
    (330:1) circle (2pt)
    (30:1) circle (2pt)
    (30:0.5) circle (2pt);
\end{tikzpicture}
}

\def\picCCC{
\begin{tikzpicture}
\node[circle,draw,thick,fill=white,inner sep=1.5pt]  (n0) at (  0:1.5) {};
\node[circle,draw,thick,fill=white,inner sep=1.5pt]  (n1) at ( 45:1.5) {};
\node[circle,draw,thick,fill=white,inner sep=1.5pt]  (n2) at ( 90:1.5) {};
\node[circle,draw,thick,fill=white,inner sep=1.5pt]  (n3) at (135:1.5) {};
\node[circle,draw,thick,fill=white,inner sep=1.5pt]  (n4) at (180:1.5) {};
\node[circle,draw,thick,fill=white,inner sep=1.5pt]  (n5) at (225:1.5) {};
\node[circle,draw,thick,fill=white,inner sep=1.5pt]  (n6) at (270:1.5) {};
\node[circle,draw,thick,fill=white,inner sep=1.5pt]  (n7) at (315:1.5) {};
\node[circle,draw,thick,fill=white,inner sep=1.5pt]  (n8) at (  0:0.8) {};
\node[circle,draw,thick,fill=white,inner sep=1.5pt]  (n9) at ( 45:0.8) {};
\node[circle,draw,thick,fill=white,inner sep=1.5pt] (n10) at ( 90:0.8) {};
\node[circle,draw,thick,fill=white,inner sep=1.5pt] (n11) at (135:0.8) {};
\node[circle,draw,thick,fill=white,inner sep=1.5pt] (n12) at (180:0.8) {};
\node[circle,draw,thick,fill=white,inner sep=1.5pt] (n13) at (225:0.8) {};
\node[circle,draw,thick,fill=white,inner sep=1.5pt] (n14) at (270:0.8) {};
\node[circle,draw,thick,fill=white,inner sep=1.5pt] (n15) at (315:0.8) {};

\draw[thick]
    (n0)--(n1)--(n2)--(n3)--(n4)--(n5)--(n6)--(n7)--(n0)
    (n0)--(n2)--(n4)--(n6)--(n0)
    (n1)--(n3)--(n5)--(n7)--(n1)
    (n0)--(n9)--(n2)--(n11)--(n4)--(n13)--(n6)--(n15)--(n0)
    (n1)--(n10)--(n3)--(n12)--(n5)--(n14)--(n7)--(n8)--(n1)
     (n8)--(n11)--(n13)-- (n8)
     (n9)--(n12)--(n14)-- (n9)
    (n10)--(n13)--(n15)--(n10)
    (n11)--(n14)-- (n8)--(n11)
    (n12)--(n15)-- (n9)--(n12)
    (n13)-- (n8)--(n10)--(n13)
    (n14)-- (n9)--(n11)--(n14)
    (n15)--(n10)--(n12)--(n15);
\end{tikzpicture}
}

\newtheorem{thm}{Theorem}[section]
\newtheorem{cor}[thm]{Corollary}
\newtheorem{lem}[thm]{Lemma}
\newtheorem{prop}[thm]{Proposition}

\newtheorem{obs}[thm]{Observation}

\newtheorem{quest}[thm]{Question}

\theoremstyle{definition}
\newtheorem{rem}[thm]{Remark}

\theoremstyle{definition}

\theoremstyle{definition}
\newtheorem{ex}[thm]{Example}

\usepackage{color}
\definecolor{red}{rgb}{.8,0,0}

\definecolor{blu}{rgb}{0,0,1}

\def\0{{\bf 0}}
\def\Fq{\mathbb{F}_{q}}
\newcommand{\bone}{\ensuremath{\mathds{1}}} 

\def\mtx#1{\begin{bmatrix} #1 \end{bmatrix}}
\def\ord#1{| #1 |} 

\newcommand{\R}{{\mathbb R}}

\newcommand{\bx}{{\bf x}}
\newcommand{\bv}{{\bf v}}

\newcommand{\by}{{\bf y}}
\newcommand{\bz}{{\bf z}}

\newcommand{\spec}{\operatorname{spec}}
\newcommand{\dspec}{\operatorname{spec}_{\D}}
\newcommand{\qD}{q_{\D}}

\newcommand{\diam }{\operatorname{diam}}
\newcommand{\Z }{\operatorname{Z}}
\newcommand{\M}{\operatorname{M}}

\newcommand{\Gc}{\overline{G}}

\newcommand{\G}{\mathcal G}
\newcommand{\D}{\mathcal D}
\newcommand{\A}{\mathcal A}

\newcommand{\tr}{\operatorname{tr}}
\newcommand{\cp}{\, \Box\,}
\newcommand{\OL}{\overline}

\newcommand{\ds}{\displaystyle}
\newcommand{\bit}{\begin{itemize}}
\newcommand{\eit}{\end{itemize}}
\newcommand{\ben}{\begin{enumerate}}
\newcommand{\een}{\end{enumerate}}
\newcommand{\beq}{\begin{equation}}
\newcommand{\eeq}{\end{equation}}
\newcommand{\bea}{\begin{eqnarray*}}
\newcommand{\eea}{\end{eqnarray*}}
\newcommand{\bean}{\begin{eqnarray}}
\newcommand{\eean}{\end{eqnarray}}
\newcommand{\bpf}{\begin{proof}}
\newcommand{\epf}{\end{proof}}
\newcommand{\x}{\times}

\newcommand{\lf}{\left\lfloor}
\newcommand{\rf}{\right\rfloor}
\newcommand{\la}{\langle}
\newcommand{\ra}{\rangle}

\title{On the distance spectra of graphs}
\author{Ghodratollah Aalipour\thanks{Department of Mathematics and Computer Sciences, Kharazmi University, 50 Taleghani St., Tehran, Iran and Department of Mathematical and Statistical Sciences, University of Colorado Denver, CO, USA (alipour.ghodratollah@gmail.com).}\and Aida Abiad\thanks{Department of Quantitative Economics, Operations Research, Maastricht University, Maastricht, The Netherlands, (aidaabiad@gmail.com).}\and Zhanar Berikkyzy\thanks{Department of Mathematics, Iowa State University, Ames, IA 50011, USA (zhanarb@iastate.edu, 
keheysse@iastate.edu, lhogben@iastate.edu, chlin@iastate.edu).}\and 
Jay Cummings\thanks{Department of Mathematics, University of California, San Diego, La Jolla, CA 92037, USA (jjcummings@math.ucsd.edu, mtait@math.ucsd.edu).}\and Jessica De Silva\thanks{Department of Mathematics, University of Nebraska-Lincoln, Lincoln, NE 68588, USA (jessica.desilva@huskers.unl.edu).}\and 
Wei Gao\thanks{Department of Mathematics and Statistics, Georgia State University, Atlanta, GA 30303, USA (wgao2@gsu.edu).}\and  Kristin Heysse\footnotemark[3]\and Leslie Hogben\footnotemark[3] \thanks{American Institute of Mathematics, 600 E. Brokaw Road, San Jose, CA 95112, USA (hogben@aimath.org).}\and Franklin H.J. Kenter\thanks{Department of Computational and Applied Mathematics, Rice University, Houston, TX 77005, USA (franklin.h.kenter@rice.edu).}\and Jephian C.-H. Lin\footnotemark[3]\and   Michael Tait\footnotemark[4]}

\begin{document}

\maketitle


\begin{abstract}
The distance matrix of a graph $G$ is the matrix containing the pairwise distances between vertices. The distance eigenvalues of  $G$ are the eigenvalues of its distance matrix and they form the distance spectrum of $G$. We determine the distance spectra of halved cubes, double odd graphs, and Doob graphs, completing the determination of distance spectra of distance regular graphs having exactly one positive distance eigenvalue.  We characterize strongly regular graphs having more positive than negative distance eigenvalues. We give examples of graphs with few distinct  distance eigenvalues but lacking regularity properties.  We also determine the determinant and inertia of the distance matrices of lollipop and barbell  graphs.
\end{abstract}



\section{Introduction}\label{sintro}

 The \emph{distance matrix} $\mathcal{D}(G)=[d_{ij}]$ of a graph $G$ is the matrix indexed by the vertices of $G$ where $d_{ij}=d(v_i, v_j)$ is the {\em distance} between the vertices $v_i$ and $v_j$, i.e., the length of a shortest path between $v_i$ and $v_j$.  Distance matrices were introduced in the study of a data communication problem in \cite{GP71}.  This problem involves finding appropriate addresses so that a message can move efficiently through a series of loops from its origin to its destination, choosing the best route at each switching point.  Recently there has been renewed interest in the loop switching problem \cite{PersiTalk}.  There has also been extensive work on distance spectra (eigenvalues of distance matrices); see \cite{AH14} for a recent survey.  

In \cite{KS1994}, the authors classify the distance regular graphs having only one positive distance eigenvalue. Such graphs  are directly related to a metric hierarchy for finite connected graphs (and more generally, for finite distance spaces, see \cite{A1984,D1960, DG1993,K1967}), which makes these graphs particularly interesting. In Section  \ref{sdistreg} we find the distance spectra of Doob graphs,  double odd graphs, and  halved cubes,  completing the determination of distance spectra of distance regular graphs that have one positive distance eigenvalue.
In Section \ref{sSRG} we characterize strongly regular graphs having more positive than negative distance eigenvalues in terms of their parameters, generalizing results in \cite{Az14}, and apply this characterization to show several additional infinite families of strongly regular graphs have this property.

Section \ref{sndeval} contains examples of graphs with specific properties and a small number of distinct distance eigenvalues. Answering a {question} in \cite{AP15}, we provide a construction for a family of connected graphs with arbitrarily large diameter that  have no more than 5 distinct distance eigenvalues but are not distance regular (Example \ref{exAP}).  We exhibit a family of graphs with arbitrarily many distinct degrees but having exactly 5 distinct distance eigenvalues (Example \ref{exmanydeg}). Finally, we give two lower bounds for the number of distinct distance eigenvalues of a graph. The first bound is for a tree, in terms of its diameter, and the second is for any graph in terms of the zero forcing number of its complement.

  In Persi Diaconis'  talk on distance spectra at the ``Connections in Discrete Mathematics: A celebration of the work of Ron Graham'' \cite{PersiTalk}, he suggested it would be worthwhile to study the distance matrix of  a clique with a path adjoined (sometimes called a lollipop graph), and in 
 Section \ref{slollipop} we determine determinants and inertias of these graphs, of barbell graphs, and of generalized barbell graphs (a family that includes both lollipops and barbells).   
The remainder of this introduction contains definitions and notation used throughout.

 All graphs are connected, simple, undirected, and finite of order at least two.  Let $G$ be a graph. The maximum distance between any two vertices in $G$ is called the \emph{diameter} of $G$ and is denoted by $\diam(G)$. Two vertices  are adjacent in  the {\em complement} of $G$, denoted by $\Gc$, if and only if they are nonadjacent in $G$.   Let $\A(G)$ denote the {\em adjacency matrix} of $G$, that is, $\A(G)=[a_{ij}]$ is the matrix indexed by the vertices of $G$ where $a_{ij}=1$ if $\{v_i,v_j\}\in E(G)$ and is $0$ otherwise.     
  The $n\x n$ all ones matrix is denoted by  $J$ and  the all ones vector by $\bone$.  
A graph $G$ is {\em regular} if every vertex has the same degree, say $k$; equivalently, $\A(G)\bone=k \bone$; observe that $k$ is the spectral radius of $\A(G)$. 

 Since $\mathcal{D}(G)$ is a real symmetric matrix,  its eigenvalues, called \emph{distance eigenvalues} of $G$, are all real. The spectrum of $\mathcal{D}(G)$ is denoted by {$\dspec(G):=\{\rho,\theta_2,\dots,\theta_n\}$ where $\rho$ is the spectral radius,} 
 and  is called the \emph{distance spectrum} of the graph $G$.

The {\em inertia} of a real symmetric matrix is the triple of integers $(n_+,n_0,n_-)$, with the entries indicating the number of positive, zero, and negative eigenvalues, respectively (counting multiplicities).  Note the order $(n_+,n_0,n_-)$, while customary in spectral graph theory, is nonstandard in linear algebra, where it is more common to use $(n_+,n_-,n_0)$.  The spectrum of a matrix can be written as a multiset (with duplicates as needed), or as a list of distinct values with the exponents ($>1$) in parentheses indicating multiplicity.


\section{Strongly regular graphs}\label{sSRG}

 A $k$-regular graph $G$ of order $n$ is {\em strongly regular} with parameters
$(n, k, \lambda,\mu)$ if every pair of adjacent vertices has $\lambda$ common neighbors
and every pair of distinct nonadjacent vertices has $\mu$ common neighbors.  For a strongly regular graph $G$ with parameters
$(n, k, \lambda,\mu)$, $\mu=0$ is equivalent to $G$ is $\frac n {k+1}$ copies of $K_{k+1}$, so we assume $\mu>0$ and thus $G$ has diameter at most 2.   There is a well known connection between  the adjacency matrix of a graph of diameter at most $2$ and its distance matrix that was exploited in \cite{Az14}.  

\begin{rem}\label{diam2} A real symmetric matrix commutes with $J$ if and only if it has constant row sum.  
Suppose $A$ commutes with $J$ and $\rho$ is the constant row sum of $A$, so $J$ and $A$ have a common eigenvector of $\mathds{1}$. Since eigenvectors of real symmetric matrices corresponding to distinct eigenvalues are orthogonal, every eigenvector of $A$ for an eigenvalue other than $\rho$ is an eigenvector of $J$ for eigenvalue 0.

 Now suppose $G$ is a graph that has diameter at most $2$.  Then $\D(G)=2(J-I)-\A(G)$.  Now suppose in addition that $G$ is regular, so  $\A(G)$ commutes with $J$.   Thus $\dspec(G)=\{2n-2-\rho\}\cup\{-\lambda-2:\lambda\in\spec(\A(G))\mbox{ and }\lambda\ne\rho\}$. 
 \end{rem}

  Let $G$ be a strongly regular graph with parameters $(n,k,\lambda,\mu)$. It is known that the (adjacency) eigenvalues of  $G$ are $\rho=k$ of multiplicity $1$, $\theta:=\frac{1}{2}(\lambda-\mu+\sqrt{(\lambda-\mu)^2+4(k-\mu)})$ of multiplicity $m_\theta:=\frac{1}{2}
(n-1-\frac{2k+(n-1)(\lambda-\mu)}{\sqrt{(\lambda-\mu)^2+4(k-\mu)}})$, and $\tau:=\frac{1}{2}(\lambda-\mu-\sqrt{(\lambda-\mu)^2+4(k-\mu)})$ of multiplicity $m_\tau:=\frac{1}{2}
(n-1+\frac{2k+(n-1)(\lambda-\mu)}{\sqrt{(\lambda-\mu)^2+4(k-\mu)}})$ \cite[Chapter 10]{GR01}.\footnote{These formulas do work for $K_n$, but in that case $m_\theta =0$.}   
Thus the distance eigenvalues of $G$ are 
   \bea\rho_{\D}:=2(n-1)-k && \mbox{ of multiplicity } 1\\
  \theta_{\D}:=-\frac{1}{2}\left(\lambda-\mu+\sqrt{(\lambda-\mu)^2+4(k-\mu)}\right)-2&& \mbox{ of multiplicity }m_\theta=\frac{1}{2}
\left(n-1-\frac{2k+(n-1)(\lambda-\mu)}{\sqrt{(\lambda-\mu)^2+4(k-\mu)}}\right)\\
\tau_{\D}:=-\frac{1}{2}\left(\lambda-\mu-\sqrt{(\lambda-\mu)^2+4(k-\mu)}\right)-2&& \mbox{ of multiplicity }m_\tau=\frac{1}{2}
\left(n-1+\frac{2k+(n-1)(\lambda-\mu)}{\sqrt{(\lambda-\mu)^2+4(k-\mu)}}\right).
\eea
For a derivation of these values using  quotient matrices, see \cite[p. 262]{AP15}.

\subsection{Optimistic strongly regular graphs}\label{sopt}

  A graph is {\em optimistic} if it has more positive than negative distance eigenvalues.  Graham and Lov\'asz raised the question of whether optimistic graphs exist (although they did not use the term).  This question was answered positively by Azarija in \cite{Az14}, where the term `optimistic' was introduced.  A strongly regular graph is a {\em conference graph} if  $(n,k,\lambda,\mu)=(n,\,\frac{n-1}2,\, \frac{n-5}4,\, \frac{n-1}4)$.
  In \cite{Az14} it is shown that conference graphs of order at least 13 are optimistic   and also that the  strongly regular graphs with parameters $(m^2, 3(m - 1),m, 6)$  are optimistic for $m \ge 5$. 
  Additional examples of optimistic strongly regular graphs, such as the Hall--Janko graph  with parameters (100,
36, 14, 12), and examples of optimistic graphs that are not strongly regular are also  presented there.

 \begin{thm}\label{SRGthm} Let $G$ be a strongly regular graph $G$ with parameters $(n,k,\lambda,\mu) $.  The graph $G$ is optimistic if and only if $\tau_{\D}>0$ and $m_\tau\ge m_\theta$. That is, $G$  is optimistic if and only if 
 \[\lambda < \frac{-4+\mu+k} 2\mbox{ and }\lambda\ge\mu-\frac {2k} {n-1}.\]
 \end{thm}
\bpf  Observe that $\theta_{\D}<0$.  Thus $G$ is optimistic if and only if $\tau_{\D}>0$ and $m_\tau\ge m_\theta$.  Simple algebra shows that  $0 < \tau_{\D}$ is equivalent to $\lambda < \frac{-4+\mu+k} 2$:  
\bea 
0 & < & -2-\frac 1 2\left(\lambda-\mu-\sqrt{(\lambda-\mu)^2+4(k-\mu)}\right)\\
4 +\lambda-\mu& < & \sqrt{(\lambda-\mu)^2+4(k-\mu)}.\eea
There are two cases.  First assume $4 +\lambda-\mu\ge 0$, so
\bea(\lambda-\mu)^2+8(\lambda-\mu)+16& < & (\lambda-\mu)^2+4(k-\mu)\\
\lambda & < & \frac{-4+\mu+k} 2.\eea
Now assume $4 +\lambda-\mu<0$, or equivalently, $\lambda-\mu<-4$.  For any strongly regular graph,  $k\ge \lambda+1$.  Thus
\bea
(\lambda-\mu)+(\lambda-k)& < & -4\\
\lambda & < & \frac{-4+\mu+k} 2.\eea

It is well known \cite[p. 222]{GR01} (and easy to see) that 
\beq\label{eq0}m_\tau-m_\theta=\frac{2k+(n-1)(\lambda-\mu)}{\sqrt{(\lambda-\mu)^2+4(k-\mu)}}.\eeq
The denominator is always positive, and thus $m_\tau\ge m_\theta$ if and only if $\lambda\ge\mu-\frac {2k} {n-1}$. 
\epf

 \begin{cor}\label{SRGparam}
Whether or not a strongly regular graph   is  optimistic depends only on its parameters  $(n,k,\lambda,\mu)$.
 \end{cor}

There are several additional families of strongly regular graphs for which the conditions in Theorem \ref{SRGthm} hold.
 
 \begin{cor}\label{a=c}
A strongly regular graph with parameters $(n,k,\mu,\mu)$ is  optimistic if and only if  $k>\mu+4$.
 \end{cor}
 \bpf  By Theorem \ref{SRGthm},  $\lambda=\mu$ implies $m_\tau>m_\theta$,  and $\tau_{\D}>0$ is equivalent to \[0<  \frac{-4+\mu+k} 2-\lambda  =\frac{1}{2} (k-\mu-4).\qedhere\]
 \epf

The family of symplectic graphs  is defined using subspaces of a vector space over a field with a finite number of elements. Let $\Fq$ be the field with $q$ elements and consider as vertices of $Sp(2m,q)$  the one  dimensional subspaces of $\Fq^{2m}$ for $m\ge 2$; let $\la\bx\ra$ denote the subspace generated by $\bx$.  The {\em alternate matrix} of order $m$ over $\Fq$ is the $2m\x 2m$ matrix $A_m=\mtx{0 & 1\\-1 & 0}\oplus\cdots\oplus \mtx{0 & 1\\-1 & 0}$ with $m$ copies of $\mtx{0 & 1\\-1 & 0}$.  The  vertices $\la\bx\ra$ and $\la\by\ra$ are adjacent in $Sp(2m,q)$ if $\bx^T A_m\by\ne 0$.  See \cite[Section 8.11]{GR01} for $Sp(2m, 2)$  and \cite{TW06} for more general $q$.
It is known that the symplectic graph $Sp(2m,q)$ is a
strongly regular graph with parameters
\[
(n, k, \lambda, \mu)=
\left(\frac{q^{2m}-1}{q-1},\, q^{2m-1}, \,q^{2m-2}(q-1),\,q^{2m-2}(q-1)\right)
\]
 (see \cite[Section 10.12]{GR01} for $q=2$ or \cite[Theorem 2.1]{TW06} for more general $q$).  The next result is immediate from Corollary \ref{a=c}. 

  \begin{cor}
The symplectic graphs $Sp(2m,q)$ are optimistic for every $q$ and $m$ except   $q=2$ and $m= 2$.
 \end{cor}
 
There are additional families of optimistic strongly regular graphs with parameters $(n,k,\mu,\mu)$.  One example is the family $O_{2m+1}(3)$ on one type of nonisotropic points, which has parameters \[\left(\frac{3^m \left(e+3^m\right)}2,\,\frac{3^{m-1} \left(3^m-e\right)}2, \, \frac{3^{m-1} \left(3^{m-1}-e\right)}2,\, \frac{3^{m-1} \left(3^{m-1}-e\right)}2\right)\] with $m\ge 2$ and $e=\pm 1$  \cite{Bro}.  

The more common definition of a {conference graph} is  a strongly regular graph with $m_\theta=m_\tau$; this is equivalent to the definition given earlier as  $(n,k,\lambda,\mu)=(n,\,\frac{n-1}2,\, \frac{n-5}4,\, \frac{n-1}4)$ \cite[Lemma 10.3.2]{GR01}. 

\begin{thm}\label{compopt} Let $G$ be a strongly regular graph with parameters $(n,k,\lambda,\mu)$.  Both $G$ and $\Gc$ are optimistic if and only if 
$G$ is a conference graph and $n\ge 13$.
\end{thm}
\bpf The parameters of $\Gc$ are $(n,\,\bar k:=n-k-1,\, \bar\lambda:=n-2-2k+\mu,\,\bar\mu:=n-2k+\lambda)$.   By Theorem \ref{SRGthm} applied to $G$ and $\Gc$, \beq\label{eq1}\lambda\ge\mu-\frac {2k} {n-1}\eeq and  
\bean
\bar \lambda&\ge& \bar\mu-\frac {2\bar k} {n-1}\nonumber\\
n-2-2k+\mu &\ge& n-2k+\lambda-\frac {2(n-k-1)} {n-1}\nonumber\\
-2+\frac {2(n-k-1)} {n-1}+\mu &\ge& \lambda\nonumber\\
\mu-\frac {2k} {n-1} &\ge& \lambda.\label{eq2}
\eean 
By comparing \eqref{eq1} and \eqref{eq2}, we see that $\lambda=\mu-\frac {2k} {n-1}$, which by \eqref{eq0} implies $m_\tau-m_\theta=0$.  Thus $G$ is a conference graph.  It is  shown in \cite{Az14} that conference graphs  are optimistic  if and only if the order is at least 13. 
\epf


\subsection{Strongly regular graphs with one positive distance eigenvalue}

Distance regular graphs (see Section \ref{sdistreg} for the definition) having one positive distance eigenvalue were studied in \cite{KS1994};  strongly regular graphs are distance regular.  Here we make some elementary observations about strongly regular graphs with one positive distance eigenvalue that will be used in Section \ref{sdistreg}.  

   \begin{prop} Let $G$ be a  strongly regular graph  with parameters $(n,k,\lambda,\mu), n\ge 3, $ and (adjacency) eigenvalues $k,\theta$, and $\tau$.  Then $G$ is a conference graph, $G=K_n$, or $\tau\le -2$.  Thus $G$ has exactly one positive distance eigenvalue if and only if 
  $G=C_5$, $G=K_n$, or $\tau=-2$. \end{prop}
\bpf Assume $G$ has exactly one positive distance eigenvalue.  Since $\tau\le\theta$ and $\tr\A(G)=0$, $\tau<0$. It is known that if a strongly regular graph $G$  is not a conference graph, then $\tau$ and $\theta$ are integers \cite[Lemma 10.3.3]{GR01}.    If $\tau=-1$, then $k=\lambda+1$ (this follows from $\tau^2-(\lambda-\mu)\tau-(k-\mu)=0$ \cite[p. 219]{GR01}),  implying $G=K_n$. Since  $\tau<0$, $\tau_D\le 0$ if and only if $\tau\ge -2$. As shown in \cite{Az14} (see also Theorem \ref{compopt}), conference graphs of order at least 13  are optimistic and so have more than one positive eigenvalue.  The conference graph on 5 vertices is $C_5$, and the conference graph on 9 vertices has $\tau=-2$.  Thus  $\tau_D\le 0$ if and only if $\tau=-2$,  $G=C_5$, 
 or $G=K_n$.
\epf
\begin{obs}\label{SRG1pos}  There are several well known families of strongly regular graphs having $\tau=-2$, and thus having one positive distance eigenvalue. Examples of such graphs and their distance spectra include:
\ben
\item The cocktail party graphs $CP(m)$ are complete multipartite graphs $K_{2,2,\dots,2}$ on $m$ partite sets of order 2; $CP(m)$ is a strongly regular graph  with parameters $(2m, 2m-2,2m-4, 2m-2)$ and has distance spectrum  $\{2m, 0^{(m-1)}, -2^{(m)}\}$.
\item The line graph $L(K_m)$ with parameters  $(\frac {m(m - 1)} 2, 2m - 4, m - 2, 4)$ has distance spectrum \\$\{(m-1)(m-2),0^{(\frac{m(m-3)}2)}, (2-m)^{(m-1)}\}$.
\item The line graph $L(K_{m,m})$ with parameters $(m^2, 2m - 2, m - 2, 2)$ has distance spectrum \\$\{2m(m-1), 0^{((m-1)^2)}, (-m)^{(2(m-1))}\}$.
\een
\end{obs}


\section{Distance regular graphs having one positive distance eigenvalue}\label{sdistreg}

Let  $i, j, k\ge 0$ be  integers. The graph $G =(V, E)$ is called {\em distance regular} if for any choice of $u, v\in V$ with $d(u, v) =k$, the number of vertices $w\in V$ such that $d(u, w) =i$ and $ d(v, w) =j$ is independent of the choice of $u$ and $v$. Distance spectra of several families of distance regular graphs were determined in \cite{AP15}.  In this section we complete the determination of the distance spectra of  all distance regular graphs having exactly one positive distance eigenvalue, as listed in  \cite[Theorem 1]{KS1994}. For individual graphs, it is simply a matter of computation, but for infinite families the determination is more challenging.  The infinite families in \cite[Theorem 1]{KS1994} are (with numbering from that paper):
 \ben
\item[(I)] cocktail party graphs $CP(m)$,
\item[(X)] cycles $C_n$ (called polygons in \cite{KS1994}), 
\item[(VII)] Hamming graphs $H(d,n)$,
\item[(VIII)] Doob graphs $D(d,n)$,
\item[(V)] Johnson graphs $J(n,r)$,
\item[(XI)] double odd graphs $DO(r)$, and
\item[(IV)] halved cubes $\frac{1}{2}Q_d$.
\een

First we summarize  the known distance spectra of  these infinite families.  
For a strongly regular graph, it is easy to determine the distance spectrum and we have listed the distance spectra of cocktail party graphs 
in Observation \ref{SRG1pos}.  The distance spectra of cycles  are determined in \cite{FCH01}, and in \cite{AP15} are presented in the following form:
For   odd  $n=2p+1$, ${ \dspec(C_n)}=\left\{\frac{n^2-1}4, (-\frac 1 4\sec^2(\frac{\pi j}n))^{(2)},j=1,\dots, p\right\}$, and for even $n=2p$ the distance eigenvalues are $\frac{n^2}4, 0^{(p-1)},(-\csc^2(\frac{\pi(2j-1)}n))^{(2)},j=1,\dots, \frac p 2$ and  $-1$ if $p$ is odd. 
Hamming graphs, whose distance spectra are determined in \cite{I2009},  are discussed in Section \ref{ssdoob} (since they are used to construct Doob graphs).
Johnson graphs, whose spectra are determined in \cite{AP15}, are  discussed in Section \ref{ssdodd} (because they are used to determine distance spectra of double odd graphs).  The remaining families are defined and their distance spectra determined in Section \ref{ssdoob} for Doob graphs, Section \ref{ssdodd} for  double odd graphs, and Section \ref{sshcube} for halved cubes.    Thus  all the infinite families  in \cite[Theorem 1]{KS1994} now have their distance spectra determined.

For completeness we list the distance spectra (some of which are known) for the individual graphs having exactly one positive distance eigenvalue; these are easily computed and we provide computational files in {\em Sage} \cite{code}.  Definitions of these graphs can be found in \cite{AH14}.

\begin{prop} The graphs listed in {\rm \cite[Theorem 1]{KS1994}} that have one positive distance eigenvalue and are not  in one of the infinite families, together with their distance  spectra, are:
 \ben
\item[(II)] the Gosset graph; $\{84, 0^{(48)}, (-12)^{(7)}\}$,
\item[(III)] the Schl\"afli graph; $\{36, 0^{(20)}, (-6)^{(6)}\}$,
\item[(VI)] the three Chang graphs (all the same spectra); $\{42, 0^{(20)}, (-6)^{(7)}\}$,
\item[(IX)] the icosahedral graph; $\left\{18,0^{(5)},(-3+ \sqrt{5})^{(3)},(-3- \sqrt{5})^{(3)}\right\}$,
\item[(XII)] the Petersen graph; $\{15, 0^{(4)}, (-3)^{(5)}\}$,
\item[(XIII)] the dodecahedral graph; $\left\{50,0^{(9)},{(-7+3 \sqrt{5})^{(3)},(-2)^{(4)},} \, (-7-3 \sqrt{5})^{(3)}\right\}$.
\een
\end{prop}


A graph $G$ is {\em transmission regular} if $\D(G)\bone=\rho\bone$ (where $\rho$ is the constant row sum of $\D(G)$). 
 Any distance regular graph is transmission regular.    Here we present some tools for transmission regular graphs and matrices constructed from distance matrices of transmission regular graphs.   We first define the {\em cartesian product} of two graphs: For graphs $G=(V,E)$ and $G'=(V',E')$ define the graph $G\cp G'$ to be the graph  whose vertex set is the cartesian product $V\x V'$ and where two vertices $(u,u')$ and $(v,v')$ are adjacent if ($u=v$ and $\{u',v'\}\in {E'}$) or ($u'=v'$ and $\{u,v\}\in E$). 
The next theorem 
is stated for distance regular graphs in \cite{I2009}, but as noted in \cite{AP15}  the proof applies to transmission regular graphs. 
\begin{thm}\label{tspectrcprod}{\rm \cite[Theorem 2.1]{I2009}} Let $G$ and $G'$ be transmission regular graphs with $\dspec(G)=\{{\rho, \theta_2},\dots,\theta_n\}$ and $\dspec(G')=\{{\rho', \theta'_2},\dots,\theta'_{n'}\}$.  Then \[\dspec(G \cp G')=\{ n' {\rho}+n {\rho'}\}\cup\{ n' \theta_2,\dots,n'\theta_n\}\cup\{ n \theta'_2,\dots,n\theta'_{n'}\} \cup \{0^{((n-1)(n'-1))}\}.\] \end{thm}


\begin{lem}\label{transregblock} Let $D$ be an $n\x n$ irreducible nonnegative symmetric matrix that commutes with $J$.  Suppose 
\[M=\mtx{a_e D +b_e J +c_e I & a_o D +b_o J +c_o I\\
a_o D +b_o J +c_o I & a_e D +b_e J +c_e I},\]
$a_e,a_o,b_e,b_o,c_e,c_o\in\R$, and $\spec(D)=\{\rho,\theta_2,\dots,\theta_{n}\}$ (where $\rho$ is the row sum so $D\bone=\rho\bone$).  Then
{\bea\spec(M)&=& \{(a_e+a_o)\rho + (b_e+b_o)n+ (c_e+c_o), (a_e-a_o)\rho + (b_e-b_o)n+ (c_e-c_o)\}\\
& &\cup\, \{(a_e+a_o)\theta_i + (c_e+c_o) : i=2,\dots,n\}\cup
\{(a_e-a_o)\theta_i + (c_e-c_o) : i=2,\dots,n\}, \eea}
where the union is a multiset union.  \end{lem}
\bpf   Since $D$ commutes with $J$, $\bone$ is an eigenvector of $D$ for eigenvalue $\rho$.  Thus $\mtx{\bone\\\bone}$ is an eigenvector for $(a_e+a_o)\rho + (b_e+b_o)n+ (c_e+c_o)$ and $\mtx{\bone\\-\bone}$ is an eigenvector for $(a_e-a_o)\rho + (b_e-b_o)n+ (c_e-c_o)$.  Let $\bx_i$ be an eigenvector for $\theta_i$ (and assume $\{\bx_2,\dots,\bx_{n}\}$ is linearly independent). 
Since eigenvectors for distinct eigenvalues are orthogonal,  $\bx_i\perp\bone$, and thus $J\bx_i=0$.  Define $\by_i=\mtx{\bx_i\\\bx_i}$ and $\bz_i=\mtx{\bx_i\\-\bx_i}$.  Then $\by_i$ is an eigenvector of $M$ for eigenvalue $(a_e+a_o)\theta_i + (c_e+c_o)$ and $\bz_i$ is an eigenvector of $M$ for eigenvalue $(a_e-a_o)\theta_i + (c_e-c_o)$.
\epf


\subsection{Hamming graphs and Doob graphs}\label{ssdoob}

A Doob graph $D(m, d)$ is the cartesian product of $m$ copies of the Shrikhande graph and the Hamming graph $H(d, 4)$.
The Shrikhande graph is the graph $S=(V,E)$ where $V=\{0, 1, 2, 3\} \x\{0, 1, 2, 3\}$ and $E=\{\{(a, b),(c, d)\}:  (a, b) -(c, d) \in\{\pm(0, 1), \pm(1, 0), \pm(1, 1)\}\}$. The distance spectrum is ${\dspec(S)}=\left\{24, 0^{(9)}, (-4)^{(6)}\right\}$ {\cite{code}}.

\begin{figure}[h!]
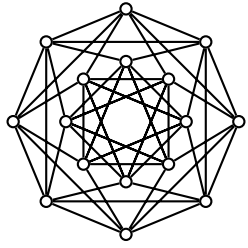

\begin{center}
\picCCC
\caption{The Shrikhande graph $S$\label{figShrik}}
\end{center}
\end{figure}

For $n\ge 2$ and $d\ge 2$, the Hamming graph $H(d,n)$ has vertex set consisting of all $d$-tuples of elements taken from $\{0,\dots,n-1\}$, with two vertices adjacent if and only if they differ in exactly one coordinate; 
$H(d,n)$ is equal to $K_n\cp \cdots\cp K_n$ with $d$ copies of $K_n$.
 In  \cite{I2009} it is shown that  the distance spectrum of the Hamming graph $H(d,n)$ is  \beq{\dspec(H(d,n))}=\left\{dn^{d-1}(n-1), 0^{(n^{d}-d(n-1)-1)}, (-n^{d-1})^{(d(n-1))}\right\}.\label{HamDspec}\eeq   
Observe that the line graph $L(K_{n,n})$ is the Hamming graph $H(2,n)$.
\begin{thm}\label{thm:doob} The distance spectrum of the Doob graph $D(m, d)$ is 
\[
\left\{3(2m+d)4^{2m+d-1}, {0^{(4^{2m+d}-6m-3d-1)},}\, (-4^{2m+d-1})^{(6m+3d)}\right\}.
\] 
\end{thm}

{\begin{proof}
Let $S_m$ be the cartesian product of $m$ copies of the Shrikhande graph. Then $\dspec(D(m, d))=\dspec(S_m \cp H(d,4))$. 
First we show that 
\beq\dspec(S_m)=\left\{6m4^{2m-1}, 0^{(16^{m}-6m-1)}, (-4^{2m-1})^{(6m)}\right\}\label{ShriDspec}\eeq
 by  induction on $m$. It is clear  that the result holds for $m=1$. Assume the result holds for $m-1$. Then by induction hypothesis, 
{$\dspec(S_{m-1})=\{6(m-1)4^{2m-3}, 0^{(16^{m-1}-6m+5)},$ $(-4^{2m-3})^{(6m-6)}\}$.}
By Theorem \ref{tspectrcprod}, we have

{\bea 
\dspec(S_{m})
&=&\dspec(S_{m-1} \cp S) \\
&=&\left\{ 16\cdot 6(m-1)4^{2m-3} + 16^{m-1}\cdot 24\right\}\\
& & \cup\ \left\{  (16\cdot0)^{(16^{m-1}-6m+5)}, (-16\cdot4^{2m-3})^{(6m-6)}\right \}\\
& &\cup\ \left\{  (16^{m-1}\cdot0)^{(9)}, (-16^{m-1}\cdot 4)^{(6)}\right\} \cup \left\{0^{(16^{m-1}-1)(16-1)}\right\}\\
&=&\left\{6m4^{2m-1}, 0^{(16^{m}-6m-1)}, (-4^{2m-1})^{(6m)}\right\}.
\eea}

From \eqref{HamDspec}, $\dspec(H(d,4))=\left\{3d4^{d-1}, 0^{(4^{d}-3d-1)}, (-4^{d-1})^{(3d)}\right\}$.
Then by \eqref{ShriDspec} and Theorem \ref{tspectrcprod}, we have
\bea 
\dspec(D(m, d))
&=&\dspec(S_m \cp H(d,4)) \\
&=&\left\{ 4^d6m4^{2m-1} + 16^{m}3d4^{d-1}\}\cup\{  (4^d \cdot 0)^{(16^{m}-6m-1)}, (-4^d4^{2m-1})^{(6m)}\right\}\\
&\ &\cup\left\{(16^m\cdot0)^{(4^{d}-3d-1)}, (-16^m4^{d-1})^{(3d)} \right\} \cup \left\{0^{(16^{m}-1)(4^d-1)}\right\}\\
&=&\left\{3(2m+d)4^{2m+d-1}, 0^{(4^{2m+d}-6m-3d-1)}, (-4^{2m+d-1})^{(6m+3d)}\right\}.
\qedhere\eea
\end{proof}}

\subsection{Johnson, Kneser,  and double odd graphs}\label{ssdodd}

Before defining Johnson graphs, we consider a more general family that includes both Johnson and Kneser graphs. For fixed $n$ and $r$, let $[n]:=\{1,\dots,n\}$ and ${[n]\choose r}$ denote the collection of all $r$-subsets of  $[n]$.    For fixed $n,r,i$, the  graph $J(n;r;i)$ is the graph defined on vertex set ${[n]\choose r}$ such that two vertices $S$ and $T$ are adjacent if and only if $|S\cap T|=r-i$.\footnote{Note the definition of  this  family of  graphs varies with the source.  Here we follow \cite{BT1984}, whereas in \cite{GR01} the  graph defined by $n, r$, and $i$ is what is here denoted by $J(n;r;r-i)$.} The Johnson graphs $J(n,r)$  are the graphs $J(n;r;1)$, and they are distance regular.  {Observe that the line graph $L(K_n)$ is the Johnson graph $J(n,2)$.  
The   distance spectrum of Johnson graphs $J(n,r)$ are determined in \cite[Theorem 3.6]{AP15}:
   \beq
{\dspec(J(n,{r}))} = \left\{s(n,r), 0^{\left({n\choose r}-n\right)},\left(-\frac {s(n,r)}{n-1}\right)^{(n-1)}\right\}\label{eq:Johnson}\eeq
where $s(n,r) =\sum _{j=0}^rj{r\choose j}{{n-r}\choose j}$.

Although they do not necessarily have one distance  positive eigenvalue and are not all distance regular,  Kneser graphs can be used to construct double odd graphs. The Kneser graph $K(n,r)=J(n;r;r)$ is the graph on the vertex set ${[n]\choose r}$ such that two vertices $S$ and $T$ are adjacent if and only if $S \cap T=\emptyset$.  Of particular interest are the {\em odd graphs} $O(r)=K(2r+1,r)$.

A double odd graph $DO(r)$ is a graph whose vertices are $r$-element or $(r+1)$-element subsets of  $[2r+1]$, where two vertices $S$ and $T$ are adjacent if and only if $S \subset T$ or $T \subset S$, as subsets.  Double odd graphs can also be constructed as tensor products of odd graphs.  We first define the {\em tensor product} of two graphs: For graphs $G=(V,E)$ and $G'=(V',E')$ define the graph $G\x G'$ to be the graph  whose vertex set is the cartesian product $V\x V'$ and where two vertices $(u,u')$ and $(v,v')$ are adjacent if $\{u,v\}\in E$ and $\{u',v'\}\in {E'}$.    To see that $DO(r)=O(r)\x P_2$, observe that $O(r)\x P_2$ has as vertices two copies of  the vertices of $O(r)=K(2r+1,r)$, call them $(S_k,1)$ and $(S_k,2)$.  Then there are no edges just between the {vertices of the form} $(S_k,1)$, no edges just between the $(S_k,2)$, and $(S_k,1)\sim (S_j,2)$ if and only if $S_k\cap S_j=\emptyset$.  Equivalently, $(S_k,1)\sim (S_j,2)$ if and only if $S_k\subset \OL{S_j}$.  We will work with the representation of $DO(r)$ as $O(r)\x P_2$.

\begin{rem}\label{Dtensor}  Let $G$ be a graph that is not bipartite.  Then $G\x P_2$ is a connected bipartite graph and $\D(G\x P_2)$ has the form $\mtx{D_e & D_o\\{D_o} & D_e}$, where $D_e$ and  $D_o$ are $n\x n$ nonnegative symmetric matrices with the entries of $D_e$ even and those of $D_o$ odd; all of these statements are obvious except the symmetry of $D_o$.  
 Observe that $D_e$ is the matrix whose $i,j$-entry is the shortest even distance between vertices $v_i$ and $v_j$ in $G$, and $D_o$ is the matrix whose $i,j$-entry is the shortest odd distance between vertices $v_i$ and $v_j$ in $G$.
\end{rem}

It is known \cite{VPV2005} that the distance between two vertices $S$ and $T$ in $K(n,r)$ is given by the formula 
\[d_{K}(S,T)=\min\left\{2\left\lceil\frac{r-|S\cap T|}{n-2r}\right\rceil, 2\left\lceil \frac{|S\cap T|}{n-2r}\right\rceil +1\right\}, \]
which for {the odd graph $O(r)$} is
\beq d_{O}(S,T)=\min\{2(r-|S\cap T|), 2|S\cap T| +1\}. \label{f1eq}\eeq
The distance between two vertices $S$ and $T$ in the Johnson graph $J(n,r)$   is given by the formula 
 \beq d_{J(n,r)}(S,T)=\frac{1}{2}|S\Delta T|=r-|S\cap T|,\label{f2eq}\eeq
 where $\Delta$ is the symmetric difference.  Let $d_J(S,T)=d_{J(2r+1,r)}(S,T)$.

\begin{prop}\label{cor:De}
{In $O(r)$, $d_O(S,T)=2$ if and only if $d_J(S,T)=1$. Furthermore, for $\D(O(r)\x P_2)=\mtx{D_e & D_o\\{D_o} & D_e}$,  $D_e= 2\D(J(2r+1,r))$}. 
\end{prop}
\begin{proof}
{The first statement follows from equations \eqref{f1eq} and \eqref{f2eq} (it also follows  from the definition). To prove the second part, let $S_0,\ldots,S_{2n}$ be a path of minimum even length between $S_0$ and $S_{2n}$, then $S_0,S_2, S_4,\ldots,S_{2n}$ is a path of length $n$  in the Johnson graph by the first statement. Conversely any path of the length $i$ in the Johnson graph between $S_0$ and $S_i$ provides  a path of length $2i$ between $S_0$ and $S_i$ in the odd graph (note that the new vertices are pairwise distinct). This implies the second statement. }
\end{proof}

\begin{prop}\label{cor:Do} {For $\D(O(r)\x P_2)=\mtx{D_e & D_o\\{D_o} & D_e}$, 
$D_o=(2r+1)J-2{\D}(J(2r+1,r))$.} 
\end{prop}
\begin{proof}
{
Let $S$ and $T$ be two vertices of {the odd graph $O(r)$}. Let $\widetilde{T}$ be an $r$-subset of $[2r+1]\setminus T$ containing $S\setminus T$. Thus $S\cap \widetilde{T}$ has the maximum size among all $r$-subsets  of $[2r+1]\setminus T$. Now the minimum odd distance between $S$ and $T$ is one more than the minimum even distance between $S$ and a neighbor of $T$ in the Kneser graph {$O(r)$}. 
By Proposition \ref{cor:De}, this is $1+2d_J(S,W)=1+2(r-|S\cap W|)$, where $W$ is a neighbor of $T$ in {$O(r)$ that}  maximizes $|S\cap W|$. It suffices to set $W=\widetilde{T}$. This implies that 
the minimum odd distance between $S$ and $T$ is 
\[1+2d_J(S,{\widetilde T})=1+2(r-|S\cap \widetilde{T}|)=2r+1-2|S\cap \widetilde{T}|=2r+1-2(|S{\setminus}T|)\]
\[=2r+1-2(|S|-|S\cap T|)=2r+1-2(r-|S\cap T|)=2r+1-2d_J(S,T).\qedhere\]
}
\end{proof}

\begin{thm}\label{thm:dodd}  Let $D$ be the distance matrix of the Johnson graph  $J(2r+1,r)$, let $J$ be the all ones matrix of order ${{2r+1}\choose r}$, and let  $r \ge 2$. 
The distance matrix  of the double odd graph $DO(r)$ is
\[  \mtx{ D_e & D_o \\ D_o & D_e} = \mtx{ 2D & (2r+1)J-2D \\ (2r+1)J-2D & 2D } \]
and its  distance spectrum    is 
\[{\dspec(DO(r))}= \left\{(2r{+}1){{2r{+}1}\choose r},  0^{\left(2{2r{+}1\choose r}-{2r-2}\right)},\left(-\frac {2s(2r{+}1,r)}{r}\right)^{({2r})}\!\!, -(2r{+}1){{2r{+}1}\choose r}+4s(2r{+}1,r)\right\}.\]
\end{thm}

\bpf
 The first statement follows from  Remark \ref{Dtensor} and Propositions \ref{cor:De} and \ref{cor:Do}.
  From Equation {\eqref{eq:Johnson}}, 
$
\spec(D) = \left\{s(2r{+}1,r), 0^{\left({2r{+}1\choose r}-{(2r{+}1)}\right)},\left(-\frac {s(2r{+}1,r)}{2r}\right)^{({2r})}\right\}$
where $s(n,r) =\sum _{j=0}^rj{r\choose j}{{n{-}r}\choose j}$.  Since $J(2r{+}1,r)$ is distance regular, it is transmission regular, that is, $D$ commutes with $J$.  Then Lemma \ref{transregblock} {establishes the result.}
\epf

We now return to {arbitrary} Kneser graphs $K(n,r)$ and determine their distance spectra.
 Let $D$ be the distance matrix of $K(n,r)$.  Let $A_0$ be the identity matrix of {order}  ${n\choose r}$ and $A_i$ the adjacency matrix of $J(n;r;i)$ for $1\leq i\leq r$.  It follows that 
\[D=\sum_{i=0}^rf(i)A_i, \text{ where } f(i)=\min\left\{2\left\lceil\frac{i}{n-2r}\right\rceil, 2\left\lceil \frac{r-i}{n-2r}\right\rceil +1\right\}.\]

It is known that $\{A_i\}_{i=0}^r$ forms an \textit{association scheme} called the \textit{Johnson scheme} and the following properties come from {the Corollary  to Theorem 2.9 and from Theorem 2.10} {in} \cite{BT1984}. For more information about association schemes, see, e.g., Section 2.3 in \cite{BT1984}.  

\begin{thm}{\rm \cite{BT1984}}
\label{thm:Johnson}
The matrices $\{A_i\}_{i=0}^r$ form a {commuting} family and are simultaneously diagonalizable.  There are subspaces $\{V_j\}_{j=0}^r$ such that 
\begin{itemize}
\item the whole space is the direct sum of $\{V_j\}_{j=0}^r$;
\item for each $A_i$, 
\[p_i(j)=\sum_{t=0}^j(-1)^t{j\choose t}{r-j\choose i-t}{n-r-j\choose i-t}\]
is an eigenvalue with multiplicity $m_j=\frac{n-2j+1}{n-j+1}{n\choose j}$ whose eigenspace is $V_j$.
\end{itemize}
\end{thm}

The value $p_i(j)=\sum_{t=0}^j(-1)^t{j\choose t}{r-j\choose i-t}{n-r-j\choose i-t}$ is known as the \textit{Eberlein polynomial}. Now with Theorem \ref{thm:Johnson} the distance spectrum of $K(n,r)$ can also be computed.

\begin{thm}
\label{thm:Kneser}
The distance spectrum of $K(n,r)$ consists of  
\[\theta_j=\sum_{i=0}^rf(i)p_i(j)\]
with multiplicity $m_j$ for $j=0,1,\ldots ,r$.
\end{thm}
\bpf
Since distance matrix of $K(n,r)$ can be written as $\sum_{i=0}^rf(i)A_i$ and $\{A_i\}_{i=0}^r$ forms a {commuting} family, the distance eigenvalues of $K(n,r)$ are also  linear combinations of the eigenvalues of $A_i$'s.  Theorem \ref{thm:Johnson} gives the common eigenspaces and the eigenvalues with multiplicity, so summing up the corresponding eigenvalues gives the distance eigenvalues of $K(n,r)$.
\epf

Note that Theorem \ref{thm:Kneser} gives us the distance spectrum of the odd graph $O(r)= K(2r+1,r)$.  

\begin{rem}
The statement in Theorem \ref{thm:Kneser} includes all the eigenvalues and multiplicities.  However, it does not guarantee $\theta_j\neq \theta_{j'}$ for different $j$ and $j'$.  If this happens, the multiplicity becomes $m_j+m_{j'}$.  
\end{rem}

\subsection{Halved cube}\label{sshcube}

Here we use a variant way of defining the hypercube $H(d,2)$.  Let $Q_d = (V,E)$ where $V = 2^{[d]}$, that is, the set of all subsets of $[d]$,  and $E = \{ \{S,T\}: |S\Delta T| = 1\}$, so $S$ and $T$ are adjacent if they differ by exactly one element. This is equivalent to the standard definition of $d$-tuples of the elements $0$ and $1$, where 1 in the $i$th position indicates $i\in S$.  Similarly, the halved cube $\frac{1}{2}Q_d = (V,E)$ has $V$ as the even (or odd) subsets of $[d]$ and $E = \{ \{S,T\}: |S\Delta T| = 2\}$. It is known that the halved cube is distance regular and has one positive distance eigenvalue \cite{KS1994}.  We  determine the  eigenvalues for the distance matrices of this graph family in Theorem \ref{thm:hcube} below.
Before beginning the proof,  we require a technical lemma.

\begin{lem} \label{simplifies} The following identities hold:
\begin{enumerate}
\item\label{s1} For $s\geq 1$, $\sum_{k=0}^s (-1)^k \binom{s}{k}=0$.
\item\label{s2} For $s\geq 2$, $\sum_{k=0}^s (-1)^k k\binom{s}{k}=0$.
\item\label{s3} For $d\ge 2$, $ \sum_{i=0}^{\lf d/2\rf} 2i \binom{d}{2 i} =d 2^{d-2}$.
\item\label{s4} For $d\ge 2$,  $\sum_{i=0}^{\lfloor (d-1)/2\rfloor}  (2i+1)\binom{d}{2i+1} =d 2^{d-2}$.
\item\label{s5}  For $d\geq 2$, $\sum_{i=0}^{\lfloor d/2\rfloor} (2i)^2 \binom{d}{2i}=d(d+1)2^{d-3}$.
\item\label{s6} For $a\geq 2, b\ge 0$, $\sum_{i=\lceil b/2\rceil}^{\lfloor (a+b)/2\rfloor} i \binom{a}{2i-b}=2^{a-3}(a+2b)$.
\end{enumerate}
\end{lem}
{\begin{proof} The first two statements are well known combinatorial results (in \cite{Br04} see: equation (5.4), and the identity $n(1+x)^{n-1}=\sum_{k=1}^nk{n\choose k}x^k$ on page 134 evaluated  at $x=-1$, respectively). For the third and fourth statements, consider the following identity (\cite[p. 134]{Br04})
\beq d(1+x)^{d-1}= \sum_{k=1}^d k\binom{d}{k}x^{k-1}. \label{eqxsimp5}\eeq
We evaluate \eqref{eqxsimp5} at $x=-1$, which yields 
\bea 0 &=& \sum_{k=1}^d k\binom{d}{k}(-1)^{k-1} \\ 
\sum_{\substack{0\le k\le d\\k\text{ even}}} k\binom{d}{k} &=& \sum_{\substack{0\le k\le d\\k\text{ odd}}} k \binom{d}{k}. 
\eea
By evaluating \eqref{eqxsimp5} at $x=1$,  we have $d(2)^{d-1}=\sum_{k=1}^{d}k \binom{d}{k}$. Since exactly half of this sum can be attributed to even $k$ and half to odd $k$, this proves statements \eqref{s3} and \eqref{s4}. A similar approach applied to \cite[Equation (5.13)]{Br04}  proves \eqref{s5}.

To prove the sixth, we can consider simplifying as follows:
\beq \sum_{i=0}^{\lfloor (a+b)/2\rfloor} i \binom{a}{2i-b} = \frac{1}{2} \sum_{i=0}^{\lfloor (a+b)/2\rfloor} (2i-b)\binom{a}{2i-b} +\frac b 2 \sum_{i=0}^{\lfloor (a+b)/2\rfloor}\binom{a}{2i-b} \label{simplify6}. \eeq

On the right hand side of equation \eqref{simplify6}, the left sum is either statement \eqref{s3} or \eqref{s4} and the right sum is  equation (5.6) or (5.7) in \cite{Br04}, both depending on the parity of $b$ and up to reindexing. We can therefore simplify \eqref{simplify6} as 
\[ \sum_{i=0}^{\lfloor (a+b)/2\rfloor} i \binom{a}{2i-b} = 2^{a-3}(a+2b). \]
 

This  completes the proof.
\end{proof}
}

For the proof of the next result, we use a technique generalized from a classical approach to calculate the adjacency eigenvalues for the hypercube. To enable this approach, rather than working with the distance matrix of the halved cube, 
$\D({\frac{1}{2} Q_d})$ (which is indexed by only the even sets or the odd sets) we consider a matrix $D$ indexed by all the subsets of $[d]$.

\begin{thm}\label{thm:hcube} The distance spectrum of the halved cube for $d\geq4$ is $\left\{d2^{d-3}, 0^{(2^{d-1}-(d+1))},(-2^{d-3})^{(d)}\right\}$. 
\end{thm}

\bpf
Define a matrix $D$  indexed by all the subsets of $[d]$ with $D_{XY} = 0$ whenever $|X|$ and $ |Y|$ differ in parity and $D_{XY}=\D(\frac{1}{2} Q_d)_{XY}=\frac{1}{2}|X\Delta Y|$ whenever $|X|$ and $|Y|$ have the same parity.  Notice that up to reordering of the rows and columns, $D=\mtx{\D({\frac{1}{2} Q_d}) & 0\\0 &\D({\frac{1}{2} Q_d})}$. Thus by finding the spectrum of $D$, we have found exactly two copies of the spectrum of $\D({\frac{1}{2} Q_d})$. 

Fix a subset $S\subseteq [d]$, 
and define a vector $\bv_S$  as follows: for any  subset $T\subseteq [d]$, the $T$-th entry of $\bv_S$ is  $\bv_S(T)=(-1)^{|S\cap T|}$.  We remark that the vectors $\bv_S$ form a complete set of orthogonal eigenvectors of the adjacency matrix of hypercube, and thus these vectors are necessarily independent. These eigenvectors can be constructed inductively using cartesian products of graphs (see, for example, \cite[Section 1.4.6]{BH}).  
We claim that for any subset $T\subseteq [d]$, 
\begin{equation} (D\bv_S)(T)=\left(\sum_{i=0}^{\lf d/2\rf}
\sum_{\substack{U\\ |U\Delta T|=2i}} i (-1)^{|(S\cap T) \Delta (S\cap U)|}\right) \bv_S(T). \label{Dv} \end{equation}
 To prove this claim, first consider an even subset $T\subseteq [d]$. Then, because the only nonzero entries of the row indexed by $T$ are those indexed by even subsets, 
\begin{eqnarray}
(D \bv_S)(T) &=& \sum_{\substack{U \subseteq [d]\\{|U| \text{ even}}}} D_{TU} (-1)^{|S \cap U|}  \nonumber \\
&=& \sum_{\substack{U \subseteq [d]\\{|U| \text{ even}}}} \frac{1}{2} |T \Delta U| (-1)^{|S \cap U|} \nonumber\\
&=& \sum_{i=0}^{\lf d/2\rf} \sum_{\substack{U\subseteq [d]\\  |T \Delta U| = 2i}} i (-1)^{|S \cap U|} \nonumber\\
&=& \label{capU} \sum_{i=0}^{\lf d/2\rf} \bigg( i \sum_{\substack{U\subseteq [d]\\  |T \Delta U| = 2i}} (-1)^{|S \cap U|} \bigg) \\
&=&\label{capT} \left[ \sum_{i=0}^{\lf d/2\rf} \bigg( i \sum_{\substack{U\subseteq [d]\\  |T \Delta U| = 2i}} (-1)^{|(S \cap T ) \Delta (S \cap U) |} \bigg) \right]  (-1)^{|S \cap T|}  
\end{eqnarray}
 where equation~\eqref{capT} follows from equation~\eqref{capU} by the fact that 
\[|S \cap U|+2|(S\cap T)\setminus (S\cap T\cap U)|=|(S\cap T)\Delta (S\cap U)|+|S\cap T|\]
and the fact that $(-1)^{2x}=1$.  Notice in the work above that the assumption that the sets were even was subsumed by the requirement that the symmetric difference with $T$ is even.  By this logic, an identical argument follows for an odd subset $T$,  which proves equation~\eqref{Dv}. 

We can consider a combinatorial interpretation of the above. Let $\mathcal{O}(d,S,T,i)$ (respectively, $\mathcal{E}(d,S,T,i)$) be the number of subsets $U \subseteq [d]$ with the same parity as $T$ such that 
\begin{itemize}
\item $|T \Delta U| = 2i$, and
\item $|(S \cap T ) \Delta (S \cap U)|$ is odd (respectively, even).
\end{itemize}

Then, we have,
\[
(D \bv_S)(T)  =\left( \sum_{i=0}^{\lf d/2\rf} i (-\mathcal{O}(d,S,T,i)+ \mathcal{E}(d,S,T,i)) \right)   \bv_S(T).
\]

We now consider the vectors associated with a few specific sets $S$. First, consider simultaneously the empty set and the set $[d]$. Notice for any $T$, we have $|(\emptyset \cap T ) \Delta (\emptyset \cap U)|=0$ and $|([d] \cap T ) \Delta ([d] \cap U)|=2i$. It follows that $\mathcal{O}(d,\emptyset,T,i)=\mathcal{O}(d,[d],T,i)=0$ and $\mathcal{E}(d,\emptyset,T,i)=\mathcal{E}(d,[d],T,i)=\binom{d}{2i}$. The last equation follows by picking $2i$ elements of $[d]$ and toggling them in $T$, which ensures the required symmetric difference. Then, by Lemma~\ref{simplifies},
\[(D \bv_\emptyset)(T)  =\left( \sum_{i=0}^{\lf d/2\rf} i \binom{d}{2 i}  \right) \bv_\emptyset(T)  = d 2^{d-3} \bv_\emptyset(T) \qquad (D \bv_{[d]})(T)  =\left( \sum_{i=0}^{\lf d/2\rf} i \binom{d}{2 i}  \right) \bv_{[d]}(T)  = d 2^{d-3} \bv_{[d]}(T)\]
so $d2^{d-3}$ is an eigenvalue of $D$  with multiplicity at least 2. Therefore $\D(\frac{1}{2}Q_d)$ has $d2^{d-3}$ as an eigenvalue.

Second, consider the singleton sets $\{s\}$ and the sets of size $d-1$, denoted $[d]\setminus\{s\}$, for all $s \in [d]$. 
We show that for any $T$, $\mathcal{O}(d,\{s\},T,i) =\mathcal{O}(d,[d]\setminus\{s\},T,i) = \binom{d-1}{2i-1} $ and $\mathcal{E}(d,\{s\},T,i) =\mathcal{E}(d,[d]\setminus\{s\},T,i)  =\binom{d-1}{2i}$. First consider the singletons. If $s \in T$, then there are $\binom{d-1}{2i}$ sets $U$ that have $s \in U$ and also the required symmetric difference. We know this because we cannot toggle $s$ out of $T$, but still must choose $2i$ to toggle from the remaining $d-1$. Notice that in this case, $|(\{s\}\cap T)\Delta (\{s\}\cap U)|=0$. Similarly, there are $\binom{d-1}{2i-1}$ sets $U$ that have $s \not\in U$ and also have the required symmetric difference. We know this because we must toggle $s$ out of $T$, and therefore can only pick $2i-1$ elements to toggle from the remaining $d-1$. Notice that in this case, $|(\{s\}\cap T)\Delta (\{s\}\cap U)|=1$. A similar argument of case analysis follows  if $s \not\in T$.  This completes the case of $S=\{s\}$. The argument is similar  for the sets of size $d-1$. 

Therefore, by Lemma~\ref{simplifies}
\bea (D\bv_{\{s\}})(T)
&=&\left(\sum_{i=0}^{\lf d/2\rf} i \left( \binom{d-1}{2i}-\binom{d-1}{2i-1}\right)\right)\bv_{\{s\}}(T) \\
&=&\left(\sum_{i=0}^{\lf d/2\rf} i \left[ \frac{d-4i}{d} \binom{d}{2i}\right]\right)\bv_{\{s\}}(T)\\
&=& \left(\sum_{i=0}^{\lf d/2\rf} i \binom{d}{2i}-\frac{4}{d}\sum_{i=0}^{\lf d/2\rf} i^2 \binom{d}{2i}\right)\bv_{\{s\}}(T)\\
&=&\left( \sum_{i=0}^{\lf d/2\rf} i \binom{d}{2i}-\frac{1}{d}\sum_{i=0}^{\lf d/2\rf} (2i)^2 \binom{d}{2i}\right)\bv_{\{s\}}(T)\\
&=&\left( d2^{d-3}-(d+1)2^{d-3}\right)\bv_{\{s\}}(T) \\
&=& -2^{d-3}\bv_{\{s\}}(T),
\eea
and similarly, \[ (D\bv_{[d]\setminus \{s\}})(T)=-2^{d-3}\bv_{[d]\setminus \{s\}}(T)\] so $-2^{d-3}$ is an eigenvalue of $D$ with multiplicity at least $2d$.  This implies $-2^{d-3}$ is an eigenvalue of $\D(\frac{1}{2}Q_d)$ with multiplicity at least $d$. 

It remains to be shown that the remaining distance eigenvalues of the halved cube are zero. To do this, we return to equation~\eqref{capT}: 
\[ (D\bv_S)(T)=\left ( \sum_{i=0}^{\lf d/2\rf} \bigg[ i \sum_{\substack{U\subseteq [d]\\ |T \Delta U| = 2i}} (-1)^{|(S \cap T ) \Delta (S \cap U) |} \bigg] \right)  (-1)^{|S \cap T|}  \]
and proceed to count the sets $U$. We now use $s$ to denote $|S|$ rather than an element.  To ensure the correct symmetric difference from $T$, we consider choosing $2i$ elements from $[d]$ and either adding them or subtracting them from $T$.  However, to determine the exponent in the inner sum, we require knowing whether any of the $2i$ elements are in the set $S$ that defines the vector.  Suppose there are $k$ such elements in $S$.  If we consider all possible values of $k$, we can rewrite the equation above as 
\[ (D\bv_S)(T)=\left(\sum_{i=0}^{\lf d/2\rf}i\sum_{k=0}^{2i} \binom{d-s}{2i-k} \binom{s}{k} (-1)^k \right) (-1)^{|S \cap T|}.\] 

It suffices to show that $\sum_{i=0}^{\lf d/2\rf}i\sum_{k=0}^{2i} \binom{d-s}{2i-k} \binom{s}{k} (-1)^k=0$ for all $s$ such that $2\leq s\leq d-2$.  
We return to the double sum and simplify using Lemma \ref{simplifies} as follows.  

\bea \sum_{i=0}^{\lf d/2 \rf}i\sum_{k=0}^{2i} \binom{d-s}{2i-k} \binom{s}{k} (-1)^k  
&=& \sum_{i=0}^{\lf d/2 \rf}i\sum_{k=0}^{s} \binom{d-s}{2i-k} \binom{s}{k} (-1)^k  \\
&=& \sum_{k=0}^s \binom{s}{k} (-1)^k \sum_{i=0}^{\lf d/2 \rf} i \binom{d-s}{2i-k} \\ 
&=& \sum_{k=0}^s \binom{s}{k} (-1)^k \sum_{i=\lceil k/2\rceil}^{\lfloor (d-s+k)/2\rfloor} i \binom{d-s}{2i-k} \\ 
&=&  \sum_{k=0}^s \binom{s}{k} (-1)^k \left( 2^{d-s-3}(d-s+2k)\right) \\ 
&=& 2^{d-s-3}\left((d-s)\sum_{k=0}^s (-1)^k \binom{s}{k} +2\sum_{k=0}^s(-1)^k k \binom{s}{k}\right)\\ 
&=&0.
\eea
This proves that the remaining eigenvalues are all zero, which completes the claimed spectrum. The alert reader may notice that by establishing \eqref{Dv} we have actually found eigenvectors for $D$; however, unlike the natural relationship between spectra, the process of  finding eigenvectors of $\D({\frac{1}{2} Q_d})$ from those of $D$ is not straightforward, and since our primary interest is in the eigenvalues we have not attempted it.
\epf


\section{Number of distinct distance eigenvalues}\label{sndeval}


In this section we construct examples of graphs having few distinct distance eigenvalues and various other properties and establish a lower bound on the number of distinct distance eigenvalues of a tree. Let $\qD(G)$ denote the number of distinct distance eigenvalues of $G$.

\subsection{Fewer distinct distance eigenvalues than diameter}\label{ssqDdiam}

We have answered the following open question:

\begin{quest}\label{AP15P4.3}
{\rm \cite[Problem 4.3]{AP15}} Are there connected graphs that are not distance regular with diameter $d$ and having less than $d +1$ distinct distance eigenvalues?
\end{quest}

\begin{ex}\label{exAP}{\rm Consider the graph $Q_d^\ell$ consisting of $Q_d$ with a leaf appended.  Figure \ref{figQ4leaf} shows $Q_4^\ell$.  Since $Q_d$ has diameter $d$, $Q_d^\ell$ has diameter $d+1$.  Since $\D(Q_d)$ has 3 distinct eigenvalues \cite[Theorem 2.2]{I2009} and the eigenvalues of $\D(Q_d^\ell)$ interlace those of $\D(Q_d)$, $\D(Q_d^\ell)$ has at most $5$ distinct eigenvalues.  For example, $\D(Q_4^\ell)$ has the 5 eigenvalues $-10.3149\dots$, $-8$ with multiplicity 3, $-1.72176\dots$, 0 with multiplicity 11, 
and $36.0366\dots$.  However $\diam(Q_4^\ell)+1 = 6$.

}\end{ex}

\begin{figure}[h!]
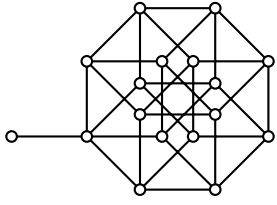

\begin{center}
\picAAA
\caption{$Q_4^\ell=Q_4$ with a leaf appended \label{figQ4leaf}}
\end{center}
\end{figure}\vspace{-20pt}

\subsection{Few distinct distance eigenvalues and many distinct degrees}\label{ssqDdeg}


The next result follows immediately from Theorem \ref{tspectrcprod}.

\begin{prop}\label{spectrcprod} If $G$ is a transmission regular graph  of order $n$ with $\dspec(G)= \{ \rho,  \theta_2,\dots,\theta_n\}$, then
\[\spec_{\D}(G \cp G)=\{ 2n \rho\}\cup\{ (n \theta_2)^{(2)},\dots,(n\theta_n)^{(2)}\} \cup \{0^{((n-1)^2)}\}.\]
\end{prop}

\begin{rem}\label{degcprod} Let $G$ and $H$ be graphs, and let  $\deg(G)$ denote the set of distinct degrees of $G$. For $x\in V(G)$ and $y\in V(H)$, $\ds \deg_{G\cp H}(x,y) = \deg_Gx+\deg_Hy$.  Thus
  $\deg({G \cp H}) = \deg(G)+\deg(H):= \{ a+b \colon a  \in \deg(G), b \in \deg(H)\}$.
\end{rem}

There is a graph $G$ with arbitrarily many distinct degrees and $\mathcal{D}(G)$ having exactly 5 distinct eigenvalues.

\begin{ex}\label{exmanydeg}
 Let $G$ be the graph in Figure \ref{figmanydeg}. Observe that $G$ is transmission regular with distance eigenvalues 
$\left\{14,\left(\frac{-5+\sqrt{33}}2\right)^{(2)},(-1)^{(4)}, \left(\frac{-5-\sqrt{33}}2\right)^{(2)}\right\}$. Furthermore, $\deg(G)=\{3,4\}$.  Let $G^i$ be $G \cp \cdots \cp G$ with $i$ copies of $G$. By Proposition \ref{spectrcprod}, $\qD(G \cp G)= 5$ and one of the distance eigenvalues is zero, and by induction $\qD(G^{2^k})=5$. By inductively applying Remark \ref{degcprod}, $G^{2^k}$ has at least $2^k+1$ distinct degrees. \end{ex}

\begin{figure}[h!]
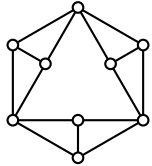

\begin{center}
\picBBB
\caption{Transmission regular graph $G$ with 4 eigenvalues that is not regular  \label{figmanydeg}}
\end{center}
\end{figure}\vspace{-20pt}

\subsection{Minimum number of distinct distance eigenvalues of a tree}\label{qDtree}

In studies of matrices or  families of matrices associated with a graph $G$, it is sometimes the case that $\diam(G)+1$ is a lower bound for the number of  distinct eigenvalues for some or all graphs $G$.  In many situations a  real symmetric $n\x n$ matrix $A$ is studied by  means of its graph $\G(A)$, which has vertices $\{1,\dots,n\}$ and edges $ij$ exactly where $a_{ij}\ne 0$.  It is well known that  a nonnegative matrix $A$ has at least $\diam(\G(A))+1$ distinct eigenvalues \cite[Theorem 2.2.1]{BR91}, and any real symmetric matrix $A$ has at least $\diam(\G(A))+1$ distinct eigenvalues if $\G(A)$ is a tree \cite[Theorem 2]{LJ02}.\footnote{ In \cite{LJ02}, $d(G)=\diam(G)+1$.}  Note that $\G(\D(T))$ is not a tree even when $T$ is a tree. 

\begin{quest}\label{qtree} 
For a tree $T$, is the number of distinct distance eigenvalues  $\qD(T)$  at least $\diam(T)+1$?   
\end{quest}

 The answer is yes for all trees of order at most 20, as determined through computations in {\em Sage} \cite{code}.  We can prove the following weaker bound.

\begin{prop}\label{treeprop}
Let $T$ be a tree.  Then $T$ has at least $\lfloor\frac{\diam(T)}{2}\rfloor$ distinct distance eigenvalues.
\end{prop}
\bpf
Let $L(T)$ be the line graph of $T$ and $D$ be the distance matrix of $T$.  Then $\diam(L(T))=\diam(T)-1$.  By \cite[Theorem 2.2.1]{BR91}, the adjacency matrix $\A(L(T))$ of $L(T)$ has at least $\diam(L(T))+1=\diam(T)$ distinct eigenvalues.  Let $K=2I+\A(L(T))$.  By \cite{M1990}, the eigenvalues of $-2K^{-1}$ interlace the eigenvalues of $D$ (with $D$ having one more).  Since $-2K^{-1}$ also has at least $\diam(T)$ distinct eigenvalues, $D$ has at least $\lceil\frac{\diam(T)}{2}\rceil$ distinct eigenvalues.
\epf


\subsection{Zero forcing bound  for the number of distinct distance eigenvalues}\label{ssZ}

 Let $G$ be a graph. Then $\mathcal{S}(G)$ denotes the family of real symmetric matrices $A$ such that $\G(A)=G$. The \textit{maximum nullity} of $G$ is defined to be the largest possible nullity among matrices in $\mathcal{S}(G)$, and is denoted by $\M(G)$. The \emph{zero forcing number} of a graph $G$, $\Z(G)$, is the minimum
cardinality of a set $S$ of blue vertices (the rest of the vertices are white) such that all the vertices are turned blue after iteratively applying the ``color-change rule'': a white vertex is turned blue if it is the only white neighbor of a blue vertex. Zero forcing was introduced in \cite{AIM08}. The zero forcing number $\Z(G)$ is used to bound the maximum nullity:

\begin{thm}
\label{thm:MlZ}
\rm{\cite{AIM08}} For any graph $G$, $\M(G)\leq \Z(G)$.
\end{thm}

\begin{thm}\label{zfdistbd}  The number of distinct eigenvalues of $\D(G)$ is at least\vspace{-5pt}
\[\frac {n-1}{\Z(\Gc)+1}+1\vspace{-5pt}\]
where $n=\ord G$.
\end{thm}
\bpf Let $\theta $ be an eigenvalue of $\D(G)$. Then the multiplicity $m_\theta$ is  $\operatorname{null}(\D(G)-\theta I)$.  Now observe that $\D(G)-J-\theta I$ is a matrix in $\mathcal{S}(\overline{G})$, so $\operatorname{null}(\D(G)-J-\theta I)\leq \M(\overline{G})\leq \Z(\overline{G})$ by Theorem \ref{thm:MlZ}.  Since $J$ is a rank-1 matrix,
\[m_\theta=\operatorname{null}(\D(G)-\theta I)\leq \operatorname{null}(\D(G)-J-\theta I)+1\leq \Z(\overline{G})+1.\] Each eigenvalue has  multiplicity at most $\Z(\overline{G})+1$.  Also, since $\D(G)$ is a irreducible matrix, by the Perron-Frobenius Theorem, the largest eigenvalue has multiplicity one.  As a consequence, the number of distinct eigenvalues is at least $\frac{n-1}{\Z(\overline{G})+1}+1$.
\epf

The bound given in Theorem \ref{zfdistbd} can be tight, as seen in the next example.
\begin{ex}
All Hamming graphs $H(d,n)$ on $d\ge 3$ dimensions (including $Q_d:=H(d,2)$) have $3$ distance eigenvalues \cite{I2009}.  Assume $d\ge 3$ and in $Q_d$ let $a:=1000\cdots 0$, $b:=0100\cdots 0$, $c:=0010\cdots 0$, $d:=1100\cdots 0$, $p:=0000\cdots 0$, and $q:=0110\cdots 0$.  Observe that $V(Q_d)\setminus \{a,b,c,d\}$ is a zero forcing set for $\overline{Q_d}$ using the following forces: $p\rightarrow d$, $d\rightarrow c$, $q\rightarrow a$, and $c\rightarrow a$.  Therefore, $\Z(\overline{Q_d})\leq |\overline{Q_d}|-4$ and $\left\lceil \frac{|{Q_d}|-1}{\Z(\overline{Q_d})+1}+1 \right\rceil=3$.    Thus the bound given in Theorem \ref{zfdistbd} is tight for hypercubes $Q_d$ with $d\geq 3$; more generally, the bound is tight for $H(d,n)$ for  $d\geq 3$.
\end{ex}


\section{Determinants and inertias of distance matrices of barbells and lollipops}\label{slollipop}

A \emph{lollipop graph}, denoted by $L(k,\ell)$ for  $k\geq 2$ and $\ell \geq 0$, is constructed by attaching a $k$-clique by an edge to one pendant vertex of a path on $\ell$ vertices.\footnote{In the literature the name `lollipop' is often used to denote a graph obtained by appending a cycle rather than a clique to the pendant vertex of the path.} A \emph{barbell graph}, denoted by $B(k,\ell)$ for $k\geq 2$ and $\ell \geq 0$, is a graph  constructed by attaching a $k$-clique by an edge to each of the two pendant vertices of a path on $\ell$ vertices.  Define the family of {\em generalized barbell graphs}, denoted by  $B(k; m; \ell)$ for $k,m\geq 2$ and $\ell \geq 0$, to be the graph constructed by attaching a $k$-clique  by an edge to one end vertex of a path on $\ell$ vertices, and attaching an $m$-clique by an edge to the other end vertex of the path.  Thus $B(k,\ell)= B(k;k;\ell)$ for $k\ge 2$ and $\ell\ge 0$, and $L(k,\ell)= B(k;2;\ell-2)$ for $\ell\ge 2$.
In this section we establish the determinant and inertia of the  distance matrices of $B(k;m;\ell)$, and hence for barbells and lollipops.

The technique we use is the method of quotient matrices, which has been applied to distance matrices previously (see, for example, \cite{AP15}).  Let $A$ be a block matrix whose rows and columns are partitioned according to some partition $X=\lbrace X_1,\ldots,X_m \rbrace$ with {\em characteristic matrix} $S$, i.e., $i,j$-entry of $S$ is 1 if $i\in X_j$ and 0 otherwise.  The \emph{quotient matrix} $B=[b_{ij}]$ of $A$ for this partition is the $m\x m$ matrix whose entries are the average row sums of the blocks of $A$, i.e. $b_{ij}$ is  the average row sum of the block $A_{ij}$ of $A$ for each $i,j$. The partition $X$ is called \emph{equitable} if the row sum of each block $A_{ij}$ is constant, that is $AS=SB$. It is well known  
that if the partition is equitable, then the spectrum of $A$ consists of the spectrum of the quotient matrix $B$ together with the eigenvalues belonging to eigenvectors orthogonal to the columns of $S$ (see, e.g.,  \cite[Lemma 2.3.1]{BH}).

\begin{thm}\label{det-barbell}  The distance determinants of generalized barbell graphs are given by
\[\det \D(B(k;m;\ell))=(-1)^{k+m+\ell-1}2^{\ell} (km (\ell+5)-2(k+ m)).\]
\end{thm}

\bpf We label the vertices of $B(k;m;\ell)$ as follows: $v_1,\ldots,v_k$ and $v_{k+1},\ldots,v_{k+m}$ are the cliques on $k$ and $m$ vertices, respectively; $v_{k+m+1},\ldots,v_{k+m+\ell}$ are the vertices on the path with $v_i$ adjacent to $v_{i+1}$; finally $v_k$ is adjacent to $v_{k+m+1}$ and $v_{k+m}$ is adjacent to $v_{k+m+\ell}$.

Let $\mathcal{Q}(B(k;m;\ell))$ be the quotient matrix of the distance matrix of $B(k;m;\ell)$, partitioning the rows and columns of $\D(B(k;m;\ell))$ according to $X=\lbrace X_1,X_2,\ldots,X_{\ell+4} \rbrace$ where $X_1=\lbrace v_1,\ldots,v_{k-1}\rbrace$, $X_{2}=\lbrace v_{k+1},\ldots,v_{k+m-1}\rbrace$, $X_3=\lbrace v_k\rbrace$, and $X_i=\lbrace v_{k+m+i-4} \rbrace$ for $i=4,\ldots,\ell+4$. Clearly, $\D(B(k;m;\ell))$ has $(k-2)+(m-2)$ eigenvectors for eigenvalue $-1$ that are orthogonal to the columns of the characteristic matrix $S$, so $\det \D(B(k;m;\ell))=(-1)^{(k-2)+(m-2)}\det \mathcal{Q}(B(k;m;\ell))=(-1)^{k+m}\det \mathcal{Q}(B(k;m;\ell))$. Thus it suffices to compute the determinant of the $(\ell+4)\x(\ell+4)$ quotient matrix.  For  $\ell=0$, simply compute the determinant of the $4\x 4$ quotient matrix (the upper left block of the block matrix in \eqref{mtx_1b} below).
Now assume $\ell\ge1$.
\bean
\det \mathcal{Q}(B(k;m;\ell)) &=& \det \left[
 \begin{array}{cccc|ccccc} \label{mtx_1b}
  k{-}2   &  (\ell{+}3)(m{-}1)&1   & \ell{+}2 & 2 & 3 &\cdots & \ell{+}1 \\
  (\ell{+}3)(k{-}1) &  m{-}2 &\ell{+}2 & 1  & \ell{+}1 & \ell &\cdots & 2 \\ 
  k{-}1     & (\ell{+}2)(m{-}1) & 0& \ell{+}1 & 1 & 2 &\cdots & \ell \\ 
 (\ell{+}2)(k{-}1) & m{-}1  & \ell{+}1& 0  & \ell   & \ell{-}1&\cdots & 1 \\ \hline
 2(k{-}1) &(\ell{+}1)(m{-}1) & 1&\ell & 0 & 1 & \cdots & \ell{-}1 \\ 
 3(k{-}1) &\ell(m{-}1) & 2&\ell{-}1 & 1 & 0 & \cdots & \ell{-}2 \\ 
  \vdots &\vdots&\vdots&\vdots&\vdots&\vdots&\ddots&\vdots\\
  (\ell{+}1)(k{-}1)  & 2(m{-}1)& \ell& 1 & \ell{-}1 & \ell{-}2 & \cdots & 0 \\
 \end{array} \right]\\[0.5em] 
 &=&
 \det \left[
 \begin{array}{cccc|cccccc} \label{mtx_2b}
 k{-}2     & (\ell{+}3)(m{-}1) & 1& \ell{+}2 & 2 & 3 &\cdots & \ell &\ell{+}1 \\
 (\ell{+}3)(k{-}1)  & m{-}2 & \ell{+}2& 1  & \ell{+}1 & \ell &\cdots &3& 2 \\ 
 1     & -(m{-}1) & {-}1& {-}1 & {-}1 & {-}1 &\cdots & {-}1&{-}1 \\
 {-}(k{-}1)     & 1 & {-}1& {-}1 & {-}1 & {-}1 &\cdots & {-}1&{-}1 \\\hline
 k{-}1     & {-}(m{-}1) & 1& {-}1 & {-}1 & {-}1 &\cdots & {-}1&{-}1 \\
 k{-}1    & {-}(m{-}1)  & 1& {-}1 &  1 & {-}1 &\cdots & {-}1&{-}1 \\ 
 k{-}1     & -(m{-}1) & 1& {-}1 &  1 &  1 &\cdots & {-}1&{-}1 \\
\vdots &\vdots&\vdots&\vdots&\vdots&\vdots&\ddots&\vdots&\vdots\\
 k{-}1     & -(m{-}1) & 1& {-}1 &  1 &  1 &\cdots & 1&{-}1 \\ 
 \end{array} \right] 
 \eean
 \bean
 &=&
 \det \left[
 \begin{array}{c|c} \label{mtx_3b}
\begin{array}{cccc} 
k-2    & (\ell{+}3)(m{-}1) & 1& \ell{+}2   \\
 (\ell{+}3)(k{-}1)  & m-2 & \ell{+}2& 1    \\ 
 1  & -(m{-}1) & {-}1& {-}1  \\
 -k & m & 0& 0  \\
 2(k{-}1) &-m & 2& 0  \\ \end{array} 
 &\begin{array}{cccc} 
 2 & 3 &\cdots &\ell{+}1 \\ 
 \ell{+}1 & \ell &\cdots & 2\\
 {-}1 & {-}1 &\cdots & {-}1\\
0 & 0 &\cdots & 0\\ 
0 & 0 &\cdots & 0\\ 
\end{array}  \\\hline
0_{(\ell{-}1)\times 4} & 2I_{\ell{-}1} \hspace{1cm} \begin{array}{c} 0\\ \vdots \\ 0\\\end{array}
\end{array} \right]\\[0.5em]
&=&
 (-2)^{\ell{-}1}\det \left[
 \begin{array}{ccccc} 
k-2    & (\ell{+}3)(m{-}1) & 1& \ell{+}2 &\ell{+}1  \\
 (\ell{+}3)(k{-}1)  & m-2 & \ell{+}2& 1 &2   \\ 
 1  & -(m{-}1) & {-}1& {-}1 &{-}1 \\
 -k & m & 0& 0  &0\\
 2(k{-}1) &-m & 2& 0  &0\\ \end{array}  \right]\nonumber\\[0.5em] 
 &=&
 ({-}1)^{\ell{-}1}2^{\ell{-}1}2 (km (\ell{+}5)-2(k{+}m)).\nonumber
\eean
The matrix (\ref{mtx_2b}) is obtained from \eqref{mtx_1b} by iteratively subtracting the $i$th row from $(i{+}1)$th row starting from  $i=\ell{+}3$ to $i=5$ and subtracting $i$th row from $(i{+}2)$th row for $i=3,2,1$. The matrix (\ref{mtx_3b}) is obtained from (\ref{mtx_2b}) by iteratively subtracting the $j$th row from $(j{+}1)$th starting from $j=\ell{+}3$ to $j=3$.  \epf


\begin{thm}
The inertia of $\D(B(k;m;\ell))$ is $(n_+,n_0,n_-)=\left(1,0,k+m+\ell-1\right)$.
\end{thm}
\bpf
We use induction on $k+m$. The base case is $k=m=2$, which is a path. Since the inertia of any tree on $n$ vertices is $(1,0,n-1)$ \cite{GP71},   the inertia of $\D(B(2;2;\ell))$ is $\left(1,0,\ell+3\right)$.  Thus the assertion follows for $k+m=4$. 
Let $\theta_1\geq\cdots\geq\theta_{k+m+\ell+1}$ be the eigenvalues of $\D(B(k+1;m;\ell))$ and $\mu_1\geq\cdots\geq\mu_{k+m+\ell}$ be the eigenvalues of $\D(B(k;m;\ell))$. By interlacing  we have $\mu_i\geq\theta_{i+1}$, for $i=1,\ldots,k+m+\ell$. The induction hypothesis implies that $\theta_3\geq\cdots\geq\theta_{k+m+\ell}$ are negative numbers and $\theta_1>0$. By Theorem~\ref{det-barbell}, the determinant will change sign, so we obtain $\theta_2<0$, completing the proof.\epf

\begin{cor}\label{det-bbarbell} The distance determinants of  barbell graphs are given by \[\det \D(B(k,\ell))=(-1)^{\ell-1}2^{\ell} (k^2 (\ell+5)-4k).\]\end{cor}

\begin{cor}\label{det-lollipop} The distance determinants of  lollipop graphs are given by  \[\det \D(L(k,\ell))=(-1)^{k+\ell-1}2^{\ell-1}\left(k(\ell+2)-2\right).\]
\end{cor}

\bpf  The case $\ell\ge 2$ follows  from Theorem \ref{det-barbell} because $L(k,\ell)=B(k;2;\ell-2)$.  The proof for $\ell=1$ is a simplified version of the proof of that theorem. The eigenvalues of the $k$-clique $L(k,0)$  are $\{k-1, (-1)^{(k-1)}\}$, so $\det \D(L(k,0))=(-1)^{k-1}(k-1)$, which agrees with $(-1)^{k+\ell-1}2^{\ell-1}\left(k(\ell+2)-2\right)$ for $\ell=0$.
\epf


\section*{Acknowledgment}  We thank  Steve Butler for stimulating discussions, encouragement and assistance, and gratefully acknowledge financial support for this research from the following grants and organizations: NSF DMS 1500662, NSF CMMI-1300477 and CMMI-1404864 (Franklin H. J. Kenter), NSF Graduate Research Fellowship Program under Grant No. DGE-1041000 (Jessica De Silva), Elsevier, International Linear Algebra Society, National Center for Theoretical Sciences  (Jephian C.-H. Lin), Iowa State University, University of Colorado Denver, University of Nebraska-Lincoln.  


\end{document}